\documentclass[letterpaper,conference]{ieeeconf}
\IEEEoverridecommandlockouts                              % This command is only
\usepackage[utf8]{inputenc}
\usepackage{textcomp}

\usepackage{graphicx}
\usepackage{tikz}
\usetikzlibrary{arrows.meta,calc,positioning}
\usepackage{amsmath}
\usepackage[version=4]{mhchem}
\usepackage{siunitx}
\usepackage{graphicx}
\usepackage{longtable,tabularx}
\usepackage[style=ieee ,sorting=none,maxbibnames=99,giveninits, doi=false]{biblatex}
\usepackage{amsmath} % assumes amsmath package installed
\usepackage{amsfonts}
\usepackage{amsthm}

\newtheorem{proposition}{Proposition}[section]

\theoremstyle{definition}

\usepackage{float} 
\usepackage[skip=0em,font=footnotesize]{caption}
\usepackage[margin = 1in]{geometry}

\usepackage{tikz}
\usetikzlibrary{shapes,arrows,calc,positioning}
\tikzstyle{bigblock} = [draw, fill=blue!20, rectangle, 
    minimum height=6em, minimum width=8em]
\tikzstyle{medblock} = [draw, fill=blue!20, rectangle, 
    minimum height=4em, minimum width=4em]    
\tikzstyle{mux} = [draw, fill=black!20, rectangle, 
    minimum height=5em, minimum width=0.1em]    
\tikzstyle{smallblock} = [draw, fill=blue!20, rectangle, 
    minimum height=2em, minimum width=3em]
    
\tikzstyle{data_block} = [draw, fill=green!20, rectangle, 
    minimum height=2em, minimum width=3em]
\tikzstyle{ops_block} = [draw, fill=blue!20, rectangle, 
    minimum height=2em, minimum width=3em]    
\tikzstyle{est_block} = [draw, fill=red!20, rectangle, 
    minimum height=2em, minimum width=3em]    
    
\tikzstyle{sum} = [draw, fill=blue!20, circle, node distance=1cm,minimum height=0.5cm]
\tikzstyle{signal} = [coordinate]
\tikzstyle{pinstyle} = [pin edge={to-,thin,black}]
\tikzstyle{block} = [draw, fill=blue!20, rectangle, 
    minimum height=3em, minimum width=9em]
\tikzstyle{blockS} = [draw, fill=blue!20, rectangle, 
    minimum height=3em, minimum width=4em]    
\tikzstyle{input} = [coordinate]
\tikzstyle{output} = [coordinate]
\usetikzlibrary{matrix}

\usetikzlibrary{positioning}
\usetikzlibrary{math}

\usepackage{hyperref}
\usepackage{xcolor}
\hypersetup{
    colorlinks,
    linkcolor={blue!100!black},
    citecolor={blue!50!black},
    urlcolor={blue!80!black}
}

\newcommand{\bc}{\begin{center}}
\newcommand{\ec}{\end{center}}
\newcommand{\benum}{\begin{enumerate}}
\newcommand{\eenum}{\end{enumerate}}
\newcommand{\nn}{\nonumber}
\newcommand{\matl}{\left[ \begin{array}}
\newcommand{\matr}{\end{array} \right]}

\renewcommand{\matl}{\begin{bmatrix}}
\renewcommand{\matr}{\end{bmatrix}}

\newcommand{\matls}{\left[ \begin{smallmatrix}}
\newcommand{\matrs}{\end{smallmatrix} \right]}
\newcommand{\isdef}{\stackrel{\triangle}{=}}

\newcommand{\BBR}{{\mathbb R}}

\usepackage{amssymb}
\usepackage{pifont}

\title{Guranteed Constrained Interception}
\title{Guaranteed Rear-Approach Interception through \\ Constrained Approach Geometry}
\title{Constraint-Lifted Sector-Constrained Guidance \\ for Rear-Approach Interception}
\author{
Jhon Manuel Portella Delgado,
Dennis S Bernstein,
and
Ankit Goel
\thanks{Jhon Manuel Portella Delgado and Dennis S. Bernstein are with the Department of Aerospace Engineering at The University of Michigan, Ann Arbor, 1320 Beal Avenue, Ann Arbor, MI 48109. {\tt\small jhonp@umich.edu, dsbaero@umich.edu}}%
\thanks{Ankit Goel is with the Department of Mechanical Engineering, University of Maryland, Baltimore County,1000 Hilltop Circle, Baltimore, MD 21250. {\tt\small ankgoel@umbc.edu }}%
}

\begin{document}

\maketitle

%%%%%%%%%%%%%%%%%%%%%%%%%%%%%%%%%%%%%%%%%%%%%%%%%%%%%%%%%%%%%%%%%%%%%%%%%%%%%%%%
\begin{abstract}
This paper develops a sector-constrained guidance (SCG) law for interception of a maneuvering target while maintaining a prescribed approach geometry.
The line of sight is required to remain within a specified angular sector behind the target throughout the engagement, resulting in a state-constrained guidance problem.
To enforce this constraint, the directional angle between the target heading and the line of sight is mapped from the prescribed sector to an unconstrained coordinate using a sigmoid diffeomorphism.
A guidance law is then designed in the lifted coordinate to generate a desired transverse velocity and corresponding pursuer lead angle.
The desired lead angle is tracked through an inner-loop turn-rate controller, which is implemented through the normal-acceleration command of a nonlinear point-mass model.
This construction embeds the sector constraint directly into the guidance law and does not require online constrained optimization.
Numerical simulations against a maneuvering target show that SCG achieves interception while maintaining the prescribed rear-sector constraint, including when the pursuer normal-acceleration command is subject to actuator saturation.
Comparisons with proportional guidance under the same engagement conditions demonstrate that proportional guidance may achieve interception but does not maintain the prescribed directional constraint.
\end{abstract}

% \keywords{one, two, three, four}
% \textit{\bf keywords:} adaptive backstepping control, dynamic extension, bicopter, multicopter.

%%%%%%%%%%%%%%%%%%%%%%%%%%%%%%%%%%%%%%%%%%%%%%%%%%%%%%%%%%%%%%%%%%%%%%%%%%%%%%%%
% \section{INTRODUCTION}

%%%% Do not delete.
% % \textit{The introduction consists of three subsections. Each subsection is essentially an answer to the following questions.}

% % \begin{enumerate}
% %     \item 
% %     What is the problem?
% %     Why is it important?
    
% %     \item 
% %     What is the state of the art in solving this problem?
% %     Provide a short literature survey. 

% %     \item 
% %     What is your approach to solving the problem?
% %     How is your approach different?
% %     What is the key contribution?

% \end{enumerate}

% List the organization of the paper. 

\section{Introduction}
Guidance of a pursuer toward a maneuvering target is a classical problem in aerospace control.
The conventional objective is interception, that is, reducing the pursuer--target separation to zero or to a prescribed capture radius.
Proportional guidance (PG) remains one of the most widely used guidance laws because of its simple structure and low computational requirements \cite{Zarchan2012,Palumbo2010}.
However, interception alone does not prescribe the direction from which the pursuer approaches the target.
In applications where the direction of approach is important, the guidance objective must therefore include both interception and a requirement on the relative engagement geometry.
The specific problem considered in this paper is to guide a pursuer to intercept a maneuvering target while maintaining the line of sight inside a prescribed angular sector behind the target throughout the engagement.
Thus, the problem is not only to achieve capture, but also to enforce a state constraint on the approach geometry.
This requirement leads to a sector-constrained interception problem, as illustrated in Figure~\ref{fig:directional_pursuit_geometry}.

Guidance laws that prescribe the geometry at interception have been
investigated extensively.
In particular, impact-angle guidance augments the interception
objective by prescribing the terminal direction from which a pursuer
or missile reaches a target.
Ratnoo and Ghose developed an impact-angle-constrained guidance law
for interception of stationary targets and derived the corresponding
feasible impact-angle conditions \cite{Ratnoo2008}.
More recently, Krstic developed pursuit and evasion feedback laws for
nonholonomic vehicles with a prescribed terminal direction of
approach, including interception from behind the evader
\cite{Krstic2026Directional}.
These formulations prescribe the approach geometry at interception,
whereas the problem considered here requires the pursuer to remain
within a prescribed rear-approach section throughout the entire engagement with
a maneuvering target.

\begin{figure}[t]
\centering
\resizebox{1\columnwidth}{!}{%
\begin{tikzpicture}[
    >=Latex,
    line cap=round,
    line join=round,
    font=\small
]

% ============================================================
% PARAMETERS (degrees, lengths)
% ============================================================

\def\thT{15}     % target flight-path angle
\def\thP{55}     % pursuer flight-path angle
\def\dlt{10}     % directional angle delta
\def\dmax{35}    % half-width of rear-approach section
\def\rng{3.8}    % range rho
\def\Rs{4.5}     % drawn radius of rear-approach section
\def\rc{0.5}     % capture radius

\pgfmathsetmacro{\phiLOS}{\thT+\dlt}
\pgfmathsetmacro{\rearL}{\thT+180-\dmax}
\pgfmathsetmacro{\rearR}{\thT+180+\dmax}

% ============================================================
% COORDINATES
% ============================================================

\coordinate (T) at (0,0);
\coordinate (P) at ($(T)+({\phiLOS+180}:\rng)$);

% ============================================================
% DIRECTIONAL SAFE SET
% ============================================================

\fill[
    green!14
]
    (T)
    --
    ++(\rearL:\Rs)
    arc[
        start angle=\rearL,
        end angle=\rearR,
        radius=\Rs
    ]
    --
    cycle;

% ============================================================
% SAFE-SET BOUNDARIES
% ============================================================

\draw[
    dashed,
    thick,
    red!65
]
    (T)
    --
    ++(\rearL:{\Rs-0.5});

\draw[
    dashed,
    thick,
    red!65
]
    (T)
    --
    ++(\rearR:{\Rs-0.5});

% ============================================================
% DIRECTIONAL-CONSTRAINT LABELS
% upper ray: pursuer above-behind target, delta = -delta_max
% lower ray: pursuer below-behind target, delta = +delta_max
% ============================================================

\node[
    anchor=south east,
    inner sep=1pt
]
    at ($(T)+(\rearL-5:2.7)$)
    {$-\delta_{\max}$};

\node[
    anchor=north west,
    inner sep=1pt
]
    at ($(T)+(\rearR:2.3)$)
    {$\delta_{\max}$};

\node[
    inner sep=1pt
]
    at ($(T)+({\thT+160}:3.4)$)
    {$|\delta|<\delta_{\max}$};

% ============================================================
% CAPTURE REGION
% ============================================================

\fill[
    gray!15
]
    (T)
    circle (\rc);

\draw[
    dashed,
    thick,
    gray!75
]
    (T)
    circle (\rc);

\draw[
    ->,
    thin
]
    (T)
    --
    ++(80:\rc);

\node[
    anchor=south,
    inner sep=1pt
]
    at ($(T)+(80:{\rc+0.05})$)
    {$\rho_{\rm c}$};

% ============================================================
% LINE OF SIGHT
% ============================================================

\draw[
    ->,
    thick
]
    (P)
    --
    (T);

\node[
    fill=green!7,
    inner sep=1pt
]
    at ($(P)!0.5!(T)+({\phiLOS+90}:0.22)$)
    {$\rho$};

% ============================================================
% TARGET VELOCITY AND FLIGHT-PATH ANGLE
% ============================================================

\draw[
    dashed,
    gray!65,
    thin
]
    (T)
    --
    ++(1.9,0);

\draw[
    ->,
    very thick
]
    (T)
    --
    ++(\thT:1.6)
    node[
        above=1pt
    ]
    {$V_{\rm T}$};

\draw[
    ->,
    thin
]
    ($(T)+(0:1.15)$)
    arc[
        start angle=0,
        end angle=\thT,
        radius=1.15
    ];

\node[
    anchor=west,
    inner sep=1pt
]
    at ($(T)+({\thT/2}:1.3)$)
    {$\theta_{\rm T}$};

% ============================================================
% PURSUER VELOCITY AND FLIGHT-PATH ANGLE
% ============================================================

\draw[
    dashed,
    gray!65,
    thin
]
    (P)
    --
    ++(2.0,0);

\draw[
    ->,
    very thick,
    blue!70!black
]
    (P)
    --
    ++(\thP:1.5)
    node[
        above left=0pt
    ]
    {$V_{\rm P}$};

\draw[
    ->,
    thin
]
    ($(P)+(0:0.7)$)
    arc[
        start angle=0,
        end angle=\thP,
        radius=0.7
    ];

\node[
    inner sep=1pt
]
    at ($(P)+(10:0.95)$)
    {$\theta_{\rm P}$};

% ============================================================
% PURSUER HEADING ERROR
% measured from the pursuer heading to the line of sight
% ============================================================

\draw[
    ->,
    thin,
    blue!70!black
]
    ($(P)+(\thP:1.1)$)
    arc[
        start angle=\thP,
        end angle=\phiLOS,
        radius=1.1
    ];

\node[
    inner sep=1pt,
    blue!70!black
]
    at ($(P)+({(\thP+\phiLOS)/2}:1.35)$)
    {$\gamma$};

% ============================================================
% LINE-OF-SIGHT ANGLE
% measured from the horizontal to the line of sight
% ============================================================

\draw[
    ->,
    thin
]
    ($(P)+(0:1.6)$)
    arc[
        start angle=0,
        end angle=\phiLOS,
        radius=1.6
    ];

\node[
    inner sep=1pt
]
    at ($(P)+(7:1.85)$)
    {$\phi$};

% ============================================================
% DIRECTIONAL ANGLE
% measured at the pursuer from the target-heading direction
% to the line of sight
% ============================================================

\draw[
    dashed,
    gray!65,
    thin
]
    (P)
    --
    ++(\thT:3.0);

\draw[
    ->,
    thick
]
    ($(P)+(\thT:2.3)$)
    arc[
        start angle=\thT,
        end angle=\phiLOS,
        radius=2.3
    ];

\node[
    inner sep=1pt
]
    at ($(P)+({\thT+\dlt/2}:2.55)$)
    {$\delta$};

% ============================================================
% PURSUER
% ============================================================

\filldraw[
    fill=blue!65,
    draw=black,
    thick
]
    (P)
    circle (0.11);

\node[
    below=4pt
]
    at (P)
    {$p_{\rm P}$};

% ============================================================
% TARGET
% ============================================================

\filldraw[
    fill=black,
    draw=black
]
    ($(T)+(-0.10,-0.10)$)
    rectangle
    ($(T)+(0.10,0.10)$);

\node[
    anchor=north west,
    inner sep=1pt
]
    at ($(T)+(-50:{\rc+0.05})$)
    {$p_{\rm T}$};

\end{tikzpicture}
}
\caption{Sector-constrained interception geometry, velocities, flight-path angles, and prescribed rear-approach section.}
\label{fig:directional_pursuit_geometry}
\end{figure}

Maintaining the engagement geometry within a prescribed sector makes
the interception problem a state-constrained control problem.
State constraints are commonly formulated in terms of forward
invariance of a prescribed safe set.
Control barrier functions (CBFs) provide a systematic framework for
enforcing forward invariance through pointwise conditions on the
control input \cite{Ames2017,Ames2019}.
For constraints with higher relative degree, additional constructions
are required to obtain conditions that can be imposed on the control
input.
Prescribed-time safety for chains of integrators and strict-feedback
nonlinear systems has also been developed by combining backstepping
with control barrier functions and quadratic programming
\cite{Abel2022,Abel2024}.
These formulations enforce the state constraint through barrier
conditions and, in the quadratic-programming implementations,
require the corresponding constrained optimization problem to be
solved online \cite{Abel2022,Abel2024}.

An alternative method for enforcing state constraints is constraint
lifting.
In this approach, a state constrained to an open interval is
represented by an unconstrained lifted coordinate through a
diffeomorphic transformation.
A controller is then designed in the lifted coordinates such that
boundedness of the lifted state implies that the original state
remains in its prescribed interval.
A sigmoid-based constraint-lifting framework for nonlinear systems
was developed in \cite{PortellaGoel2026Constraint}, where the
resulting explicit control law guarantees asymptotic stability and
forward invariance of the prescribed constraint set.
The framework has subsequently been extended to nonlinear systems
with parametric uncertainty \cite{PortellaGoel2026Adaptive} and to
discrete-time nonlinear systems \cite{PortellaGoel2026Discrete}.
These results motivate the use of constraint lifting for
sector-constrained interception, since the prescribed rear-approach section can
be represented directly as an interval constraint on the relative
engagement angle.

In this paper, constraint lifting is used to construct a
sector-constrained guidance law for interception of a maneuvering
target using nonlinear point-mass dynamics.
The angle between the target heading and the line of sight is
constrained to a prescribed interval that defines the admissible rear
sector.
A sigmoid diffeomorphism maps this constrained directional error to
an unconstrained coordinate, and the resulting lifted dynamics are
used to construct the desired transverse velocity of the pursuer.
The corresponding desired lead angle is tracked through the pursuer
turn-rate dynamics and implemented through the normal-acceleration
input of the nonlinear point-mass model.
The resulting guidance architecture therefore combines an outer
constraint-regulation loop with an inner lead-angle tracking loop.
Capture is defined by the pursuer entering a prescribed neighborhood
of the target, while the directional constraint is required to remain
satisfied throughout the engagement until capture.

The main contribution of this paper is the development of an explicit
sector-constrained guidance law for interception of a maneuvering
target using constraint lifting.
Unlike terminal impact-angle guidance, the prescribed approach-sector constraint is enforced throughout the engagement rather than only at capture.
The constraint is embedded directly in the lifted coordinate, and the
resulting guidance law does not require an online quadratic program.
Numerical simulations using nonlinear point-mass dynamics with a
maneuvering target and actuator-limited normal acceleration are used
to evaluate both interception and constraint satisfaction.
The proposed SCG law is compared with conventional PG under the same
physical engagement conditions.

The paper is organized as follows.
Section~\ref{sec:problem_formulation} formulates the sector-constrained
interception problem and defines the prescribed rear-sector constraint.
Section~\ref{sec:method} develops the constraint-lifted guidance law
and its implementation using nonlinear point-mass dynamics.
Section~\ref{sec:numerical_simulation} presents numerical simulations
and comparisons with proportional guidance.
Finally, Section~\ref{sec:conclusions} summarizes the results and
discusses directions for future work.

% \newpage
% \clearpage
\section{Sector-Constrained Interception Problem}
\label{sec:problem_formulation}

\subsection{Engagement Dynamics}

The pursuer and target are modeled using the planar nonlinear point-mass dynamics developed in \cite{DorseyGoel2025}.
For $i\in\{{\rm P},{\rm T}\}$, let $p_i\in\BBR^2$ denote the position, $V_i>0$ the speed, and $\theta_i$ the flight-path angle, that is, the angle from the horizontal axis to the velocity vector.
The dynamics are
\begin{align}
    \dot p_i
    &=
    V_i
    \begin{bmatrix}
        \cos\theta_i\\
        \sin\theta_i
    \end{bmatrix},
    \label{eq:position_dynamics}
    \\
    \dot V_i
    &=
    \frac{T_i-D_i}{m_i}
    -
    g\sin\theta_i,
    \label{eq:speed_dynamics}
    \\
    \dot\theta_i
    &=
    -\frac{
        n_{z,i}+g\cos\theta_i
    }{
        V_i
    },
    \label{eq:heading_dynamics}
\end{align}
where $m_i$, $T_i$, $D_i$, and $n_{z,i}$ denote the mass, thrust, drag, and normal acceleration, respectively, and $g$ denotes the gravitational acceleration.
The flight-path-angle rate is denoted by $\omega_i\isdef\dot\theta_i$.

\subsection{Relative Engagement Dynamics}

Let $r\isdef p_{\rm T}-p_{\rm P}$ denote the relative position from the pursuer to the target, let $\rho\isdef\|r\|$ denote the range, and let $\phi\isdef\operatorname{atan2}(r_2,r_1)$ denote the line-of-sight angle.
Define the directional angle by
\begin{align}
    \delta
    \isdef
    \operatorname{wrap}
    \left(
        \phi-\theta_{\rm T}
    \right),
    \label{eq:directional_angle}
\end{align}
where $\operatorname{wrap}(\alpha)$ denotes the principal-angle wrapping operation\footnote{
$\operatorname{wrap}(\alpha)\isdef
\operatorname{atan2}(\sin\alpha,\cos\alpha)$.
},
and define the lead angle, that is, the angle from the pursuer velocity vector to the line of sight, by
\begin{align}
    \gamma
    \isdef
    \operatorname{wrap}
    \left(
        \phi-\theta_{\rm P}
    \right).
    \label{eq:pursuer_heading_error}
\end{align}
The engagement geometry is shown in Figure~\ref{fig:directional_pursuit_geometry}.
Since the line of sight points from the pursuer to the target, $\delta=0$ implies $\phi=\theta_{\rm T}$, which corresponds to the pursuer located directly behind the target along the direction defined by the target flight-path angle.

Differentiating $r$ along \eqref{eq:position_dynamics} yields the range and line-of-sight dynamics
\begin{align}
    \dot\rho
    &=
    V_{\rm T}\cos\delta
    -
    V_{\rm P}\cos\gamma,
    \label{eq:relative_range_dynamics}
    \\
    \dot\phi
    &=
    \frac{
        V_{\rm P}\sin\gamma
        -
        V_{\rm T}\sin\delta
    }{
        \rho
    }.
    \label{eq:los_rate}
\end{align}
Since $\dot\delta=\dot\phi-\omega_{\rm T}$ and $\dot\gamma=\dot\phi-\omega_{\rm P}$ away from the discontinuities of the wrapping operation,
\begin{align}
    \dot\delta
    &=
    \frac{
        V_{\rm P}\sin\gamma
        -
        V_{\rm T}\sin\delta
    }{
        \rho
    }
    -
    \omega_{\rm T},
    \label{eq:directional_angle_dynamics}
    \\
    \dot\gamma
    &=
    \frac{
        V_{\rm P}\sin\gamma
        -
        V_{\rm T}\sin\delta
    }{
        \rho
    }
    -
    \omega_{\rm P}.
    \label{eq:pursuer_heading_error_dynamics}
\end{align}

\subsection{Control Problem}

Let $\rho_{\rm c}>0$ denote the capture radius, and define the capture time by $t_{\rm c}\isdef\inf\{t\geq0:\rho(t)\leq\rho_{\rm c}\}$.
Let $\delta_{\max}\in(0,\pi/2)$ define the prescribed rear-approach section
\begin{align}
    |\delta(t)|
        <
    \delta_{\max}.
    \label{eq:directional_constraint}
\end{align}
The initial conditions satisfy $\rho(0)>\rho_{\rm c}$ and $|\delta(0)|<\delta_{\max}$.
The pursuer and target states $p_i$, $V_i$, and $\theta_i$ and the target turn rate $\omega_{\rm T}$ are available for feedback.
The control problem is to design the normal-acceleration command $n_{z,{\rm P}}$ such that $t_{\rm c}<\infty$ and \eqref{eq:directional_constraint} holds for all $t\in[0,t_{\rm c}]$.
No intermediate rear-following distance is prescribed, and the pursuer approaches the target from within the rear-approach section until capture.

% \section{Constraint-Lifted Interception Guidance}
\section{Sector-Constrained Guidance}
\label{sec:method}

This section develops the sector-constrained guidance law for interception of a maneuvering target.
The directional constraint is mapped to an unconstrained lifted coordinate, and a desired lead angle is synthesized to stabilize this coordinate.
Since this desired lead angle cannot be differentiated using the available feedback, it is tracked through a command filter and realized via a normal-acceleration command.
The resulting closed-loop guidance law is then shown to satisfy the directional constraint and produce closing motion toward capture.

\subsection{Directional Constraint Lifting}

To enforce the directional constraint \eqref{eq:directional_constraint} without recourse to online constrained optimization, $\delta$ is mapped to an unconstrained coordinate through a diffeomorphic transformation, so that a guidance law can instead be designed to keep the transformed coordinate bounded.

Define the directional safe set
\begin{align}
    \mathcal{S}_{\delta}
    \isdef
    \left\{
        \delta\in\BBR:
        |\delta|<\delta_{\max}
    \right\},
    \label{eq:directional_safe_set}
\end{align}
so that \eqref{eq:directional_constraint} is equivalent to $\delta(t)\in\mathcal{S}_{\delta}$.
Define the normalized directional angle $\chi\isdef\delta/\delta_{\max}$, so that $\delta\in\mathcal{S}_{\delta}$ if and only if $|\chi|<1$.
The lifted directional coordinate is defined by
\begin{align}
    z
    \isdef
    \delta_{\max}
    \operatorname{atanh}
    \left(
        \frac{\delta}{\delta_{\max}}
    \right).
    \label{eq:lifted_directional_coordinate}
\end{align}
The map \eqref{eq:lifted_directional_coordinate} is a diffeomorphism from $\mathcal{S}_{\delta}$ onto $\BBR$, with inverse
\begin{align}
    \delta
    =
    \delta_{\max}
    \tanh
    \left(
        \frac{z}{\delta_{\max}}
    \right).
    \label{eq:inverse_directional_transformation}
\end{align}
This map is an instance of the sigmoid-based constraint-lifting transformations developed in \cite{PortellaGoel2026Constraint}.
Since $\tanh$ maps $\BBR$ into $(-1,1)$, every finite $z$ corresponds to $\delta\in\mathcal{S}_{\delta}$, and $|z|\rightarrow\infty$ as $|\delta|\rightarrow\delta_{\max}$.
Hence, if $z$ remains finite on $[0,t_{\rm c}]$, then \eqref{eq:directional_constraint} holds on $[0,t_{\rm c}]$.
Furthermore, $|z|\leq\bar z$ implies $|\delta|\leq\delta_{\max}\tanh(\bar z/\delta_{\max})<\delta_{\max}$.
The scaling by $\delta_{\max}$ in \eqref{eq:lifted_directional_coordinate} yields $\partial z/\partial\delta=1$ at $\delta=0$, so that $z\approx\delta$ near the center of the rear-approach section.

\subsection{Sector-Constrained Guidance Synthesis}
Having established that boundedness of $z$ is sufficient for the directional constraint to hold, a guidance law is now designed to stabilize $z$.
Differentiating \eqref{eq:lifted_directional_coordinate} and using \eqref{eq:directional_angle_dynamics} yields
\begin{align}
    \dot z
    =
    \frac{1}{
        1-\chi^2
    }
    \left[
        \frac{
            V_{\rm P}\sin\gamma
            -
            V_{\rm T}\sin\delta
        }{
            \rho
        }
        -
        \omega_{\rm T}
    \right].
    \label{eq:lifted_directional_dynamics}
\end{align}
The normal-acceleration command $n_{z,{\rm P}}$ does not appear in \eqref{eq:lifted_directional_dynamics}.
Instead, substituting \eqref{eq:heading_dynamics} into \eqref{eq:pursuer_heading_error_dynamics} yields
\begin{align}
    \dot\gamma
    =
    \dot\phi
    +
    \frac{
        n_{z,{\rm P}}+g\cos\theta_{\rm P}
    }{
        V_{\rm P}
    },
    \label{eq:lead_angle_dynamics_nz}
\end{align}
so that $n_{z,{\rm P}}$ enters the lead-angle dynamics and $\gamma$ enters the lifted dynamics \eqref{eq:lifted_directional_dynamics}.
The engagement thus has the cascade structure $n_{z,{\rm P}}\rightarrow\gamma\rightarrow z$, in which $\gamma$ acts as a virtual input to \eqref{eq:lifted_directional_dynamics}.
The guidance law is designed in two steps.
This subsection determines a desired lead angle $\gamma_{\rm d}$ that yields prescribed lifted dynamics, and Section~\ref{subsec:lead_angle_tracking} designs $n_{z,{\rm P}}$ so that $\gamma$ tracks a filtered version of $\gamma_{\rm d}$.
This construction follows the dynamic surface control and command-filtered backstepping approaches in \cite{swaroop2000dynamic,farrell2009command}.

The quantity $V_{\rm P}\sin\gamma$ is the component of the pursuer velocity normal to the line of sight.
Treating $V_{\rm P}\sin\gamma$ as a virtual control input to \eqref{eq:lifted_directional_dynamics}, let $k_z>0$, and define the desired normal velocity
\begin{align}
    \nu_{\rm d}
    \isdef
    V_{\rm T}\sin\delta
    +
    \rho\omega_{\rm T}
    -
    \rho k_z
    \left(
        1-\chi^2
    \right)z.
    \label{eq:desired_virtual_control}
\end{align}
Physically, $\nu_{\rm d}$ is the value the pursuer's normal velocity $V_{\rm P}\sin\gamma$ must take for the lifted state to decay according to \eqref{eq:desired_lifted_dynamics}.
The first two terms of \eqref{eq:desired_virtual_control} cancel the target-induced terms in \eqref{eq:lifted_directional_dynamics}, and the third term provides feedback on $z$.
If $V_{\rm P}\sin\gamma=\nu_{\rm d}$, then \eqref{eq:lifted_directional_dynamics} becomes
\begin{align}
    \dot z
    =
    -k_z z.
    \label{eq:desired_lifted_dynamics}
\end{align}
Since $\gamma$, not $V_{\rm P}\sin\gamma$ directly, is the quantity commanded through \eqref{eq:lead_angle_dynamics_nz}, the virtual control target $\nu_{\rm d}$ must be inverted to obtain a desired lead angle.
Provided that
\begin{align}
    |\nu_{\rm d}|
    <
    V_{\rm P},
    \label{eq:virtual_control_admissibility}
\end{align}
the desired lead angle is defined by
\begin{align}
    \gamma_{\rm d}
    \isdef
    \arcsin
    \left(
        \frac{\nu_{\rm d}}{V_{\rm P}}
    \right).
    \label{eq:desired_gamma}
\end{align}
The equation $V_{\rm P}\sin\gamma=\nu_{\rm d}$ has two solutions in $(-\pi,\pi]$, and the principal branch in \eqref{eq:desired_gamma} yields $\cos\gamma_{\rm d}>0$.
Consequently, the term $-V_{\rm P}\cos\gamma_{\rm d}$ in \eqref{eq:relative_range_dynamics} is negative and decreases the range.

Since $\gamma\neq\gamma_{\rm d}$ in general, \eqref{eq:lifted_directional_dynamics} is written as
\begin{align}
    \dot z
    =
    -k_z z
    +
    \frac{V_{\rm P}}{
        \rho\left(1-\chi^2\right)
    }
    \left(
        \sin\gamma-\sin\gamma_{\rm d}
    \right).
    \label{eq:perturbed_lifted_dynamics}
\end{align}
Thus, \eqref{eq:desired_lifted_dynamics} holds only if $\gamma=\gamma_{\rm d}$, and the second term in \eqref{eq:perturbed_lifted_dynamics} is the perturbation induced by the lead-angle tracking error, which Proposition~\ref{prop:constraint_satisfaction} bounds.

\subsection{Lead-Angle Tracking and Normal-Acceleration Control}
\label{subsec:lead_angle_tracking}

The constraint-lifted synthesis provides the desired lead angle
$\gamma_{\rm d}$ in \eqref{eq:desired_gamma}.
Direct tracking of $\gamma_{\rm d}$ would require its time derivative
$\dot\gamma_{\rm d}$ in the normal-acceleration command.
Differentiating \eqref{eq:desired_gamma} yields
\begin{align}
    \dot\gamma_{\rm d}
    =
    \frac{1}{
        \sqrt{1-(\nu_{\rm d}/V_{\rm P})^2}
    }
    \left(
        \frac{\dot\nu_{\rm d}}{V_{\rm P}}
        -
        \frac{\nu_{\rm d}\dot V_{\rm P}}{V_{\rm P}^2}
    \right),
    \label{eq:gamma_d_dot}
\end{align}
where, from \eqref{eq:desired_virtual_control},
\begin{align}
    \dot\nu_{\rm d}
    &=
    \dot V_{\rm T}\sin\delta
    +
    V_{\rm T}\cos\delta\,\dot\delta
    +
    \dot\rho\,\omega_{\rm T}
    \nn \\ &\quad 
    +
    \rho\,\dot\omega_{\rm T}
    -
    \rho k_z
    \frac{d}{dt}
    \left[
        \left(1-\chi^2\right)z
    \right].
    \label{eq:nu_d_dot}
\end{align}
Thus, $\dot\gamma_{\rm d}$ depends on the target acceleration
$\dot V_{\rm T}$ and angular acceleration $\dot\omega_{\rm T}$, and on
the pursuer acceleration $\dot V_{\rm P}$, none of which is available
for feedback under the measurement set in
Section~\ref{sec:problem_formulation}.
To avoid requiring these derivatives, $\gamma_{\rm d}$ is passed through
the first-order command filter
\begin{align}
    \dot\gamma_{\rm c}
    =
    \lambda
    \left(
        \gamma_{\rm d}
        -
        \gamma_{\rm c}
    \right),
    \label{eq:heading_command_filter}
\end{align}
where $\lambda>0$ and $\gamma_{\rm c}(0)=\gamma_{\rm d}(0)$.
Thus, $\gamma_{\rm c}$ and $\dot\gamma_{\rm c}$ are available directly
from the filter without differentiating $\gamma_{\rm d}$.

Define the lead-angle-tracking error as $e_\gamma\isdef\gamma-\gamma_{\rm c}$.
% Differentiating yields
% \begin{align}
%     \dot e_\gamma
%     =
%     \dot\gamma
%     -
%     \dot\gamma_{\rm c}.
%     \label{eq:heading_tracking_error_def}
% \end{align}
%
Since $\gamma=\phi-\theta_{\rm P}$ away from the discontinuities of the wrapping operation, $\dot\gamma=\dot\phi-\dot\theta_{\rm P}$, and thus
\begin{align}
    \dot e_\gamma
    =
    \dot\phi
    -
    \dot\theta_{\rm P}
    -
    \dot\gamma_{\rm c}.
    \label{eq:heading_tracking_error}
\end{align}
%
% The nominal error dynamics $\dot e_\gamma=-k_\gamma e_\gamma$, where
% $k_\gamma>0$, are obtained by requiring the right-hand side of
% \eqref{eq:heading_tracking_error} to equal $-k_\gamma e_\gamma$, that is,
% \begin{align}
%     \dot\phi
%     -
%     \dot\theta_{\rm P}
%     -
%     \dot\gamma_{\rm c}
%     =
%     -k_\gamma e_\gamma.
%     \label{eq:nominal_error_condition}
% \end{align}
%
% Solving \eqref{eq:nominal_error_condition} for $\dot\theta_{\rm P}$ yields the 
If pursuer turn rate satisfies 
\begin{align}
    \dot\theta_{\rm P}
    =
    \dot\phi
    -
    \dot\gamma_{\rm c}
    +
    k_\gamma e_\gamma,
    \label{eq:required_pursuer_turn_rate}
\end{align}
then
\begin{align}
    \dot e_\gamma=-k_\gamma e_\gamma.
    \label{eq:e_gamma_dynamics}
\end{align}
Letting $k_\gamma>0$ ensures $e_\gamma(t) =e_\gamma(0)e^{-k_\gamma t} \to 0.$

Finally, equating \eqref{eq:required_pursuer_turn_rate} with the pursuer flight-path-angle dynamics \eqref{eq:heading_dynamics} yields the normal-acceleration command
\begin{align}
    n_{z,{\rm P}}
    =
    -V_{\rm P}
    \left(
        \dot\phi
        -
        \dot\gamma_{\rm c}
        +
        k_\gamma e_\gamma
    \right)
    -
    g\cos\theta_{\rm P}.
    \label{eq:pursuer_normal_acceleration_command}
\end{align}
The commanded normal acceleration is assumed to be realized instantaneously.
Substitution of \eqref{eq:pursuer_normal_acceleration_command} into the pursuer flight-path-angle dynamics therefore yields \eqref{eq:e_gamma_dynamics}.
Figure~\ref{fig:block_diagram} summarizes the resulting cascade structure, in which $n_{z,{\rm P}}$ drives $\gamma$ through the inner lead-angle-tracking loop, and $\gamma$ in turn drives $z$ through the outer constraint-lifted synthesis loop.
% \begin{align}
%     \dot e_\gamma
%     =
%     -k_\gamma e_\gamma,
%     \label{eq:heading_tracking_error_dynamics}
% \end{align}
% and hence
% $e_\gamma(t)=e_\gamma(0)e^{-k_\gamma t}$.

\begin{figure*}[h]
\centering
% \resizebox{2\columnwidth}{!}
{%
\begin{tikzpicture}[
    >=Latex,
    line cap=round,
    line join=round,
    font=\small,
    block/.style={
        draw,
        thick,
        minimum height=1.1cm,
        minimum width=2.1cm,
        align=center,
        fill=blue!4
    },
    sum/.style={
        draw,
        thick,
        circle,
        inner sep=0pt,
        minimum size=0.5cm
    },
    cascade/.style={
        ->,
        thick,
        blue!60!black
    },
    fbk/.style={
        ->,
        thick,
        gray!70
    },
    feedforward/.style={
    ->,
    thick,
    blue!60!black,
    dashed
},
]

% ============================================================
% MAIN CASCADE:  gamma_d -> gamma_c -> e_gamma -> n_{z,P} -> gamma -> z
% ============================================================

\node[block] (synth)
    {Constraint-Lifted\\Synthesis \\ \eqref{eq:desired_virtual_control}, \eqref{eq:desired_gamma}};

\node[block, right=0.5cm of synth] (filter)
    {Command \\ Filter \\ \eqref{eq:heading_command_filter}};

\node[sum, right=0.5cm of filter] (sum1) {};

\node[block, right=0.5cm of sum1] (control)
    {Normal-Accel.\\Command \\ \eqref{eq:pursuer_normal_acceleration_command}};

\node[block, right=0.7cm of control] (plant)
    {Engagement\\Dynamics \\ \eqref{eq:heading_dynamics}, \eqref{eq:relative_range_dynamics}--\eqref{eq:los_rate}};

\node[block, below=1.3cm of plant] (lift)
    {Constraint\\Lifting \\ \eqref{eq:lifted_directional_coordinate}};

% ============================================================
% FORWARD (CASCADE) PATH
% ============================================================

\draw[cascade]
    (synth) -- node[above, font=\scriptsize]{$\gamma_{\rm d}$}
    (filter);

\draw[cascade]
    (filter) -- node[above, font=\scriptsize]{$\gamma_{\rm c}$}
    (sum1);
% \node[font=\scriptsize] at ($(sum1)+(-0.22,0.08)$) {$-$};

\draw[cascade]
    (sum1) -- node[above, font=\scriptsize]{$e_\gamma$}
    (control);

\draw[cascade]
    (control) -- node[above, font=\scriptsize]{$n_{z,{\rm P}}$}
    (plant);

\draw[cascade]
    (plant.south) -- ++(0,-0.55)
    node[pos=0, right, font=\scriptsize]{}
    -| node[pos=0.02, left, font=\scriptsize]{$\delta$}
    (lift.north);

% ============================================================
% INNER FEEDBACK: gamma -> sum1  (e_gamma = gamma - gamma_c)
% ============================================================

\draw[fbk]
    (plant.north) |- ++(0,0.5)
    -| node[pos=0.15, above, font=\scriptsize]{$\gamma$}
    (sum1.north);
% \node[font=\scriptsize] at ($(sum1)+(0.1,0.22)$) {$+$};

% \dot\gamma_c also enters the control block
% dot gamma_c is a direct feedforward signal from the filter to the control block, not feedback
% dot gamma_c enters the control block directly, alongside phidot
\draw[cascade]
    (filter.south) |- ++(0,-0.55) -| (control.250);
\node[font=\scriptsize] at ($(control.250)+(-0.5,-0.35)$) {$\dot\gamma_{\rm c}$};

% ============================================================
% OUTER FEEDBACK: z (and delta) -> synthesis block
% ============================================================

\draw[fbk]
    (lift.west) -| ++(-5,0)
    -| node[pos=0.25, above, font=\scriptsize]{$z$}
    (synth.south);

% ============================================================
% EXOGENOUS TARGET INPUTS
% ============================================================

\node[left=0.4cm of synth, font=\scriptsize, align=center]
    (target) {$V_{\rm T}$,\\ $\omega_{\rm T}$};
\draw[->, thick]
    (target) -- (synth.west);

\node[below=1.3cm of control, font=\scriptsize, align=center]
    (phidot) {$\dot\phi$};
\draw[->, thick]
    (phidot) -- (control.south);

\end{tikzpicture}
}
\caption{Block diagram of the sector-constrained guidance law, showing the cascade structure $n_{z,{\rm P}}\rightarrow\gamma\rightarrow z$ (highlighted path) with inner lead-angle tracking and outer constraint-lifted synthesis feedback loops.}
\label{fig:block_diagram}
\end{figure*}
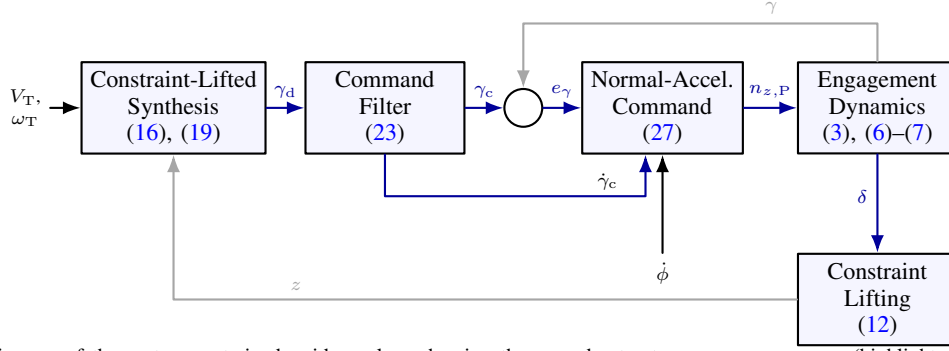

% \begin{remark}
% The pursuer thrust is prescribed as $T_{\rm P}=T_{{\rm P},0}$, and the pursuer speed evolves according to the longitudinal point-mass dynamics.
% The engagement is assumed to satisfy $V_{\rm P}(t)>V_{\rm T}(t)$ throughout the engagement.
% Thus, $n_{z,{\rm P}}$ is the only guidance control input.
% \end{remark}

\subsection{Constraint Satisfaction Under Filtered Tracking}
\label{subsec:constraint_satisfaction}

The synthesis in Section~\ref{sec:method} establishes that
$\dot z=-k_z z$ when $\gamma=\gamma_{\rm d}$ exactly.
Since $\gamma_{\rm d}$ is not directly tracked, but is instead
approximated by the command-filtered $\gamma_{\rm c}$, the actual
lead-angle tracking error $\gamma-\gamma_{\rm d}$ is nonzero in
general, as shown in \eqref{eq:perturbed_lifted_dynamics}.
The following proposition shows that, despite this filtering-induced
perturbation, the directional constraint \eqref{eq:directional_constraint}
remains satisfied throughout the engagement, provided the
lead-angle-tracking and command-filter gains are chosen sufficiently
large relative to the perturbation.

\begin{proposition}
\label{prop:constraint_satisfaction}
Consider the closed-loop system under the sector-constrained guidance law
and command filter \eqref{eq:heading_command_filter}.
Assume that, for $t\in[0,t_{\rm c}]$,
\[
    |\dot\gamma_{\rm d}(t)|\leq\bar d_\gamma,
    \qquad
    V_{\rm P}(t)\leq\overline V_{\rm P}.
\]
Let $\bar z>0$ and suppose that $|z(0)|<\bar z$.
If
\begin{align}
    \frac{\overline V_{\rm P}}{\rho_{\rm c}}
    \cosh^2
    \left(
        \frac{\bar z}{\delta_{\max}}
    \right)
    \left(
        |e_\gamma(0)|
        +
        \frac{\bar d_\gamma}{\lambda}
    \right)
    <
    k_z\bar z,
    \label{eq:filtered_constraint_condition}
\end{align}
then $|z(t)|<\bar z$ for all $t\in[0,t_{\rm c}]$ and, consequently,
\[
    |\delta(t)|
    <
    \delta_{\max}
    \tanh
    \left(
        \frac{\bar z}{\delta_{\max}}
    \right)
    <
    \delta_{\max}.
\]
\end{proposition}

\begin{proof}
    See Appendix \ref{app:proof_constraint_satisfaction}.
\end{proof}

Condition \eqref{eq:filtered_constraint_condition} is a sufficient condition ensuring that the stabilizing term in the lifted-coordinate dynamics dominates the perturbation introduced by command filtering and lead-angle tracking.
The condition is conservative and, since it depends on the a priori unknown quantities $\bar d_\gamma$, $\overline V_{\rm P}$, and $e_\gamma(0)$, is generally difficult to verify exactly before an engagement.
Its role is instead to establish existence: it shows that, for sufficiently large $k_\gamma$ and $\lambda$ relative to the target maneuver bound $\bar d_\gamma$, gains can always be chosen so that the filtering-induced perturbation does not violate the directional constraint.

In particular, the pursuer speed is bounded over the finite engagement, $\rho\geq\rho_{\rm c}>0$ prior to capture, and restricting $|z|\leq\bar z$ bounds the gain associated with the constraint-lifting transformation.
Furthermore, increasing the lead-angle-tracking gain $k_\gamma$ and the command-filter bandwidth $\lambda$ reduces the deviation between $\gamma$ and $\gamma_{\rm d}$.
Thus, \eqref{eq:filtered_constraint_condition} formalizes the requirement that the flight-path-angle dynamics be sufficiently fast relative to the constraint-lifted outer-loop dynamics.

\subsection{Closing Motion Toward Capture}
\label{subsec:closing_motion}

In addition to the directional constraint, the control problem in Section~\ref{sec:problem_formulation} requires capture in finite time, that is, $t_{\rm c}<\infty$.
The following proposition gives a sufficient condition under which the range $\rho$ is strictly
decreasing.

\begin{proposition}
\label{prop:closing_motion}
Consider the closed-loop system under the sector-constrained guidance
law and command filter \eqref{eq:heading_command_filter}.
If $V_{\rm P}(t)\cos\gamma(t)>V_{\rm T}(t)$ for $t\in[0,t_{\rm c}]$, then for $t\in[0,t_{\rm c}],$
\begin{align}
    \dot\rho(t)
    =
    V_{\rm T}(t)\cos\delta(t)
    -
    V_{\rm P}(t)\cos\gamma(t)
    <
    0
    .
    \label{eq:closing_rate}
\end{align}
\end{proposition}

\begin{proof}
Since $\cos\delta(t)\leq1$, \eqref{eq:relative_range_dynamics} gives
$\dot\rho(t)\leq V_{\rm T}(t)-V_{\rm P}(t)\cos\gamma(t)<0$ by hypothesis.
\end{proof}

Proposition~\ref{prop:closing_motion} shows that positive closing
motion follows once $\gamma$ remains near $\gamma_{\rm d}$, since
$\cos\gamma_{\rm d}>0$ under \eqref{eq:virtual_control_admissibility}.
A rigorous guarantee that $\cos\gamma(t)>0$ holds throughout the
engagement, and a corresponding bound on the finite capture time
$t_{\rm c}$, would require quantifying the lead-angle tracking error
against the admissibility margin in \eqref{eq:virtual_control_admissibility};
this analysis is left to future work.
In the numerical study of Section~\ref{sec:numerical_simulation}, the closing velocity $v_{\rm c} \isdef -\dot\rho$ remains positive throughout the engagement, and the pursuer achieves capture in finite time.

% \clearpage

\section{Numerical Simulation}
\label{sec:numerical_simulation}

This section numerically evaluates the proposed sector-constrained guidance (SCG) law using the nonlinear point-mass engagement dynamics developed in Section~\ref{sec:problem_formulation}.
The numerical study is designed to assess whether SCG achieves interception of a maneuvering target while maintaining the prescribed rear-approach constraint throughout the engagement.
For comparison, proportional guidance (PG) is evaluated under the same physical engagement conditions.

The engagement is initialized with pursuer position
$p_{\rm P}(0)=[0\;\;0]^{\rm T}\,\mathrm{m}$
and target position
$p_{\rm T}(0)=[5000\;\;300]^{\rm T}\,\mathrm{m}$,
which correspond to an initial pursuer--target separation of
$\rho(0)\approx5009\,\mathrm{m}$.
The initial target speed and flight-path angle are
$V_{\rm T}(0)=584\,\mathrm{m/s}$ and
$\theta_{\rm T}(0)=10^\circ$, respectively.
The initial pursuer speed and flight-path angle are
$V_{\rm P}(0)=900\,\mathrm{m/s}$ and
$\theta_{\rm P}(0)=0^\circ$, respectively.
The corresponding initial line-of-sight and directional angles are
$\phi(0)\approx3.43^\circ$ and
$\delta(0)\approx-6.57^\circ$, respectively.

To evaluate the guidance law against a maneuvering target, the target is commanded to follow a time-varying turn-rate profile,
\begin{align}
    \omega_{\rm T}(t)
        =
        0.080
        +
        0.0005\sin(0.85t)
        +
        0.008\cos(1.70t).
    \label{eq:numerical_target_turn_rate}
\end{align}
The superposition of the constant and oscillatory components produces a continuously varying target flight-path angle throughout the engagement.

The pursuer and target are simulated using the nonlinear point-mass dynamics described in Section~\ref{sec:problem_formulation}, with physical parameters selected consistently with~\cite{DorseyGoel2025}.
The target mass is $m_{\rm T}=10000\,\mathrm{kg}$ and its thrust is held constant at $T_{\rm T}=50000\,\mathrm{N}$.
The pursuer mass is $m_{\rm P}=204\,\mathrm{kg}$ and its thrust is held constant at $T_{\rm P}=15000\,\mathrm{N}$.
The simulations use gravitational acceleration $g=9.81\,\mathrm{m/s^2}$.

The prescribed rear-sector constraint is
$|\delta(t)|<\delta_{\max}$, where
$\delta_{\max}=20^\circ$.
% Thus, the engagement is initialized inside the admissible sector and relatively close to its lower boundary.
Interception is defined using the capture radius
$\rho_{\rm c}=15\,\mathrm{m}$.
For both SCG and PG, the pursuer normal-acceleration command is limited to
$|n_{z,{\rm P}}|\leq10g$.

For SCG, the command filter is initialized with
$\gamma_{\rm c}(0)=\gamma_{\rm d}(0),$ where $\gamma_{\rm d}(0)$ can be calculated using \eqref{eq:desired_virtual_control} and \eqref{eq:desired_gamma}.
The controller parameters are selected as
$k_z=0.50$, $\lambda=0.55$, and $k_\gamma=10$.
As shown in Figure~\ref{fig:block_diagram}, the guidance law has a cascade structure in which the lead-angle tracking dynamics form an inner loop and the constraint-lifted dynamics form an outer loop.
The gain $k_z$ therefore sets the response of the outer constraint-regulation dynamics, whereas $k_\gamma$ determines the convergence rate of the inner lead-angle-tracking loop.
Accordingly, $k_\gamma$ is selected substantially larger than $k_z$ so that the lead angle tracks the filtered command on a faster time scale than the directional constraint evolves.
The filter bandwidth $\lambda$ is selected to smooth the desired lead-angle command while retaining sufficiently fast command tracking.

For comparison, PG is implemented using the same nonlinear point-mass dynamics, initial conditions, target maneuver, and normal-acceleration limit.
Two navigation constants, $N_{\rm PG}=1.0$ and $N_{\rm PG}=2$, are considered.
Unlike SCG, the PG formulation does not explicitly account for the prescribed directional constraint.

Figure~\ref{fig:trajectory} compares the planar engagement trajectories for SCG and PG.
The SCG pursuer intercepts the maneuvering target while approaching from within the prescribed rear-approach section.
For $N_{\rm PG}=1.0$, PG reduces the range but does not achieve interception, with the triangle marking the point of closest approach.
For $N_{\rm PG}=2$, PG achieves interception but does not maintain the prescribed directional constraint.

Figure~\ref{fig:states} shows the corresponding directional angle, range, lead-angle tracking, and normal-acceleration command.
Under SCG, the directional angle remains within the prescribed bound
$|\delta(t)|<20^\circ$ throughout the engagement.
Under SCG, the directional angle remains within the prescribed bound
$|\delta(t)|<20^\circ$ throughout the engagement and remains within
the prescribed bounds until capture.
In contrast, both PG trajectories leave the prescribed directional sector.

\begin{figure}[h]
    \centering
    \includegraphics[width=\columnwidth]{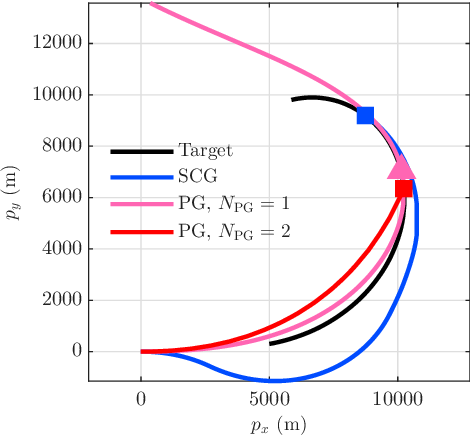}
    \caption{
    Planar trajectories of the target, SCG pursuer, and PG pursuer. 
    The SCG guidance law achieves interception while satisfying the directional constraint. 
    For $N_{\rm PG}=1$, PG does not achieve interception, with the triangle indicating the point of closest approach. 
    For $N_{\rm PG}=2$, PG achieves interception but violates the directional constraint.
    }
    \label{fig:trajectory}
\end{figure}

The range histories show that SCG reaches the prescribed capture radius
$\rho_{\rm c}=15\,\mathrm{m}$.
PG with $N_{\rm PG}=2$ also reaches the capture radius, whereas the $N_{\rm PG}=1.0$ trajectory reaches a minimum separation and then moves away from the target.
Although PG can achieve interception for a suitable navigation constant, neither PG case maintains the prescribed directional constraint.

The SCG lead-angle histories show the desired lead angle $\gamma_{\rm d}$, the filtered command $\gamma_{\rm c}$, and the actual lead angle $\gamma$.
The actual lead angle follows the filtered command throughout the engagement, while the command filter smooths the desired lead-angle trajectory.
To represent a physical actuator constraint, the pursuer normal-acceleration command is saturated at
$|n_{z,{\rm P}}|\leq10g$.
As shown in Figure~\ref{fig:states}, the commanded acceleration reaches this saturation limit during portions of the engagement.
Despite this actuator saturation, SCG maintains the prescribed directional constraint and achieves capture in the simulated engagement.

\begin{figure}
    \centering
    \includegraphics[width=\columnwidth]{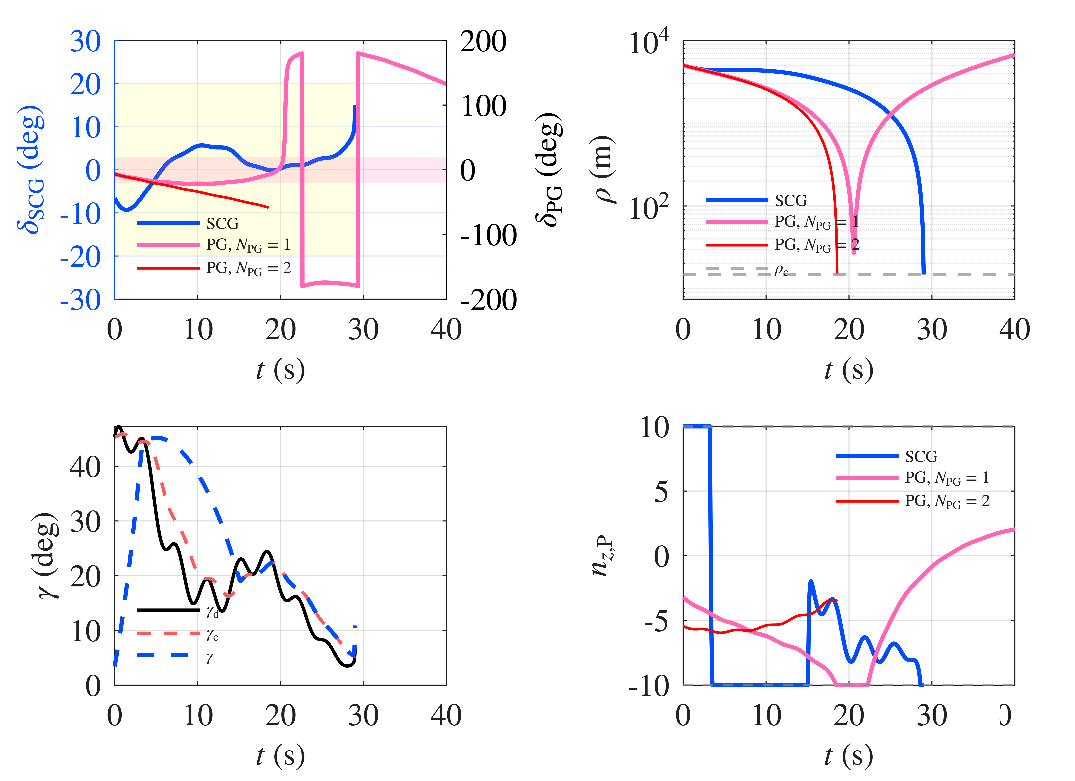}
    \caption{
    Directional angle, range, SCG lead-angle tracking, and normal-acceleration command. 
    The SCG guidance law maintains the directional angle within the prescribed safe region and achieves interception. 
    PG with $N_{\rm PG}=1$ does not achieve interception, whereas PG with $N_{\rm PG}=2$ achieves interception while violating the directional constraint. 
    The SCG lead angle tracks the commanded and desired lead angles while maintaining the normal-acceleration command within the prescribed limits.
    }
    \label{fig:states}
\end{figure}

Overall, the numerical results show that SCG achieves capture while maintaining the prescribed rear-sector constraint, including when the normal-acceleration command is subject to actuator saturation.

\section{Conclusions}
\label{sec:conclusions}

This paper developed a sector-constrained guidance (SCG) law for interception of a maneuvering target while maintaining a prescribed rear-approach sector.
The directional constraint was enforced by mapping the constrained engagement angle to an unconstrained lifted coordinate and designing the guidance law in the lifted space.
The resulting architecture combines an outer constraint-regulation loop with an inner lead-angle tracking loop and does not require online constrained optimization.
Numerical simulations using nonlinear point-mass dynamics showed that SCG achieves interception while maintaining the prescribed sector constraint, including in the presence of normal-acceleration saturation.
Comparisons with proportional guidance showed that proportional guidance may achieve interception but does not maintain the prescribed directional constraint.
Future work will consider extensions to time-varying directional constraints and additional actuator limitations.

\appendix

\section{Appendix}

\subsection{Proportional Guidance Baseline}
\label{app:PG}

For comparison with the sector-constrained guidance law, proportional
guidance is implemented using the same nonlinear point-mass pursuer
dynamics and the same initial conditions and target maneuver used in the
numerical study.

Let $N_{\rm PG}>0$ denote the proportional guidance gain.
The pursuer turn-rate command is
\begin{align}
    \omega_{{\rm P},{\rm PG}}
    \isdef
    N_{\rm PG}\dot\phi.
    \label{eq:pg_turn_rate_command}
\end{align}
Equating $\dot\theta_{\rm P}=\omega_{{\rm P},{\rm PG}}$ with the pursuer
flight-path-angle dynamics \eqref{eq:heading_dynamics} yields the
corresponding normal-acceleration command
\begin{align}
    n_{z,{\rm P},{\rm PG}}
    \isdef
    -V_{\rm P}
    \omega_{{\rm P},{\rm PG}}
    -
    g\cos\theta_{\rm P}.
    \label{eq:pg_normal_acceleration}
\end{align}

Unlike the sector-constrained guidance law, proportional guidance does not
explicitly incorporate the directional constraint
$|\delta|<\delta_{\max}$.
It is therefore used as a baseline to distinguish successful interception
from interception that also satisfies the prescribed rear-approach
geometry.

\subsection{Proof of Proposition~\ref{prop:constraint_satisfaction}}
\label{app:proof_constraint_satisfaction}

Define the command-filter error as
$
    e_{\rm c}
    \isdef
    \gamma_{\rm c}-\gamma_{\rm d}.
$
From \eqref{eq:heading_command_filter},
$
    \dot e_{\rm c}
    =
    -\lambda e_{\rm c}
    -
    \dot\gamma_{\rm d}.
$
Since $\gamma_{\rm c}(0)=\gamma_{\rm d}(0)$,
$e_{\rm c}(0)=0$.
Using
$|\dot\gamma_{\rm d}(t)|\leq\bar d_\gamma$ yields
\[
    |e_{\rm c}(t)|
    \leq
    \frac{\bar d_\gamma}{\lambda}
    \left(
        1-e^{-\lambda t}
    \right)
    \leq
    \frac{\bar d_\gamma}{\lambda}.
\]

From \eqref{eq:e_gamma_dynamics},
$
    |e_\gamma(t)|
    =
    |e_\gamma(0)|e^{-k_\gamma t}
    \leq
    |e_\gamma(0)|.
$
Since
$\gamma-\gamma_{\rm d}=e_\gamma+e_{\rm c}$,
\[
    |\gamma-\gamma_{\rm d}|
    \leq
    |e_\gamma(0)|
    +
    \frac{\bar d_\gamma}{\lambda}.
\]

Using the definition of the virtual control, the actual lifted-coordinate
dynamics are
\[
    \dot z
    =
    -k_z z
    +
    \frac{V_{\rm P}}{
        \rho(1-\chi^2)
    }
    \left(
        \sin\gamma-\sin\gamma_{\rm d}
    \right).
\]
Since
$|\sin a-\sin b|\leq|a-b|$,
\[
    |\sin\gamma-\sin\gamma_{\rm d}|
    \leq
    |e_\gamma(0)|
    +
    \frac{\bar d_\gamma}{\lambda}.
\]

From \eqref{eq:inverse_directional_transformation},
$\chi=\tanh(z/\delta_{\max})$.
Hence, for $|z|\leq\bar z$,
\[
    \frac{1}{1-\chi^2}
    =
    \cosh^2
    \left(
        \frac{z}{\delta_{\max}}
    \right)
    \leq
    \cosh^2
    \left(
        \frac{\bar z}{\delta_{\max}}
    \right).
\]
Furthermore,
$\rho(t)\geq\rho_{\rm c}$ for
$t\in[0,t_{\rm c}]$ and
$V_{\rm P}(t)\leq\overline V_{\rm P}$.
Therefore,
\begin{align}
    &\left|
        \frac{V_{\rm P}}{
            \rho(1-\chi^2)
        }
        \left(
            \sin\gamma-\sin\gamma_{\rm d}
        \right)
    \right|
    \nn \\
    &\leq
    \frac{\overline V_{\rm P}}{\rho_{\rm c}}
    \cosh^2
    \left(
        \frac{\bar z}{\delta_{\max}}
    \right)
    \left(
        |e_\gamma(0)|
        +
        \frac{\bar d_\gamma}{\lambda}
    \right).
    \nn
\end{align}

By \eqref{eq:filtered_constraint_condition}, at $z=\bar z$,
$\dot z<0$, whereas at $z=-\bar z$,
$\dot z>0$.
Thus, the vector field points inward at both boundaries of
$[-\bar z,\bar z]$.
Since $|z(0)|<\bar z$, the interval
$[-\bar z,\bar z]$ is forward invariant, and therefore
$|z(t)|<\bar z$ for all $t\in[0,t_{\rm c}]$.

Finally, from \eqref{eq:inverse_directional_transformation},
\begin{align}
    |\delta(t)|
    &=
    \delta_{\max}
    \left|
        \tanh
        \left(
            \frac{z(t)}{\delta_{\max}}
        \right)
    \right|
    <
    \delta_{\max}
    \tanh
    \left(
        \frac{\bar z}{\delta_{\max}}
    \right)
    \nn \\
    &<
    \delta_{\max}.
    \nn
\end{align}
Thus, the directional constraint is satisfied throughout the engagement.
\hfill$\square$

\printbibliography
\end{document}